\documentclass{elsarticle}
\usepackage{amsmath}
\usepackage{amsfonts}
\usepackage[all]{xy}
\usepackage{multirow}

\newcommand{\Hom}{\ensuremath{\operatorname{Hom}}}
\newcommand{\Endl}{\ensuremath{\operatorname{End}^{l}}}
\newcommand{\Endr}{\ensuremath{\operatorname{End}^{r}}}
\newcommand{\Kmod}{$K$-\textbf{Mod}}
\newproof{rmk}{Remark}
\newdefinition{definition}{Definition}[section]
\newtheorem{lemma}{Lemma}[section]
\newdefinition{example}{Example}[section]
\newproof{proof}{Proof}
\newtheorem{theorem}{Theorem}[section]
\newtheorem{corollary}{Corollary}
\begin{document}

\title{A Bialgebraic Approach to Automata and Formal Language Theory}
\author{James Worthington}
\address{Mathematics Department, Malott Hall, Cornell University, Ithaca, NY
14853-4201 USA}
\ead{worthing@math.cornell.edu}

\begin{abstract}
A bialgebra is a structure which is simultaneously an algebra and a coalgebra, 
such that the algebraic and coalgebraic parts are compatible.  Bialgebras are 
usually studied over a commutative ring.  In this paper, we apply the defining 
diagrams of algebras, coalgebras, and bialgebras to categories of semimodules
and semimodule homomorphisms over a commutative semiring.  We then treat
automata as certain representation objects of algebras and formal languages as elements
of dual algebras of coalgebras.  Using this perspective, we demonstrate many
analogies between the two theories. Finally, we show that there is an adjunction 
between the category of ``algebraic'' automata and the category of deterministic
automata.  Using this adjunction, we show that {\em $K$-linear automaton morphisms} 
can be used as the sole rule of inference in a complete proof system for
automaton equivalence.
\end{abstract}
\maketitle
\section{Introduction}
Automata and formal languages are fundamental objects of study in theoretical 
computer science.  Classically, they have been studied from
an algebraic perspective, focusing on transition matrices of automata, 
algebraic operations defined on formal power series, etc., as in the 
Kleene-Sch\"{u}tzenberger theorem.  More recently, automata have been studied 
from a coalgebraic perspective, focusing on the co-operations of transition 
and observation, and the coalgebraic notion of bisimulation.  See, for 
example, \cite{bib:aacaeic}.

In this paper, we treat automata and formal languages from a {\em bialgebraic}
perspective: one that includes both algebraic and coalgebraic structures, with 
appropriate interactions between the two.  This provides a rich framework to 
study automata and formal languages; using bialgebras, we can succinctly
express operations on automata, operations on languages, maps between automata,
language homomorphisms, and the interactions among them.  We then show that 
automata as representation objects of algebras are related to the standard
notion of a deterministic automaton via an adjunction. 

A note on terminology: there are two uses of the word ``coalgebra'' in the
literature we reference. In an algebra course, one would define ``coalgebra'' as
a variety containing a counit map and the binary operations of addition and
comultiplication; i.e., the formal dual of an algebra (in the ``vector space
with multiplication'' sense). In computer science literature, the word
``coalgebra'' can refer to arbitrary $F$-coalgebras for a given endofunctor $F$ of
{\bf Set}: so-called ``universal coalgebra'' \cite{bib:ucatos}.  Except for
Section \ref{section:determinize} below, our coalgebras are the more specific 
``algebra course'' kind.

While bialgebras are usually studied over a commutative ring $R$, it is
desirable to work over semirings when studying automata and formal languages.
Hence we must define a tensor product for semimodules
over a semiring; we show that a tensor product with the correct universal
property exists when the semiring in question is commutative.  Semimodules over
a semiring are in general not as well-behaved as vector spaces (neither are
modules over a ring). However, free semimodules exist, and have all the useful 
properties that freeness entails.  We remark that we treat input words as
elements of free semimodules, and that the standard definition of a weighted
automaton employs a free semimodule on a finite sets of states.

We then proceed by defining a bialgebra $B$ on the set of all finite words over
an alphabet $\Sigma$.  The algebraic operation of multiplication describes how
to ``put words together''; it is essentially concatenation.  The coalgebraic operation of 
comultiplication, a map $B \rightarrow B \otimes B$, describes how to
``split words apart''; there are several comultiplications of interest.

Given an algebra $A$, we are interested in the structures on which A acts,
i.e., its {\em representation objects}.  We can encode an automaton as a
representation object of an algebra $A$ equipped with a start state and an
observation function. These automata compute elements of the dual module of $A$, which we view as formal 
languages. Automaton morphisms, i.e., linear maps between automata which
preserve the language accepted, are shown to be instances of {\em linear intertwiners}. 
Given a coalgebra $C$, the dual module of $C$ also corresponds to a set of
languages. A standard result is that a comultiplication on a coalgebra defines 
a multiplication on the dual module.  For appropriate bialgebras, these two
views of formal languages interact nicely, and we can use a bialgebra 
construction to ``run two automata in parallel.''

Finally, we show that determinizing an automaton
is essentially forgetting the semimodule structure on its states.
This idea is made precise with functors between categories of algebraic
automata and categories of deterministic automata.  Each
category has its own advantages: algebraic automata can be combined in 
useful ways, and can be nondeterministic, while deterministic automata have
unique minimizations.  An adjunction between these two categories allows
us to prove that a proof system for algebraic automata equivalence is complete;
the rules of inference are automaton morphisms.  This generalizes the proof
system treated explicitly in \cite{bib:apgka} and implicitly in \cite{bib:ko94}
to arbitrary semirings.  

Other authors have explored the role of bialgebras in the theory of automata and formal 
languages. In \cite{bib:bar} and \cite{bib:riob}, Grossman and Larson study the question 
of which elements of the dual of a bialgebra can be represented by the
action of the bialgebra on a finite object and prove the Myhill-Nerode theorem using notions 
from the theory of algebras. Our definition of an automaton is a straightforward generalization
of theirs.  In \cite{bib:ddlawm} and \cite{bib:sdsc}, Duchamp et al. examine
rationality-preserving operations of languages defined using various comultiplications 
on the algebra of input words, and construct the corresponding automata.  They also apply 
these ideas to problems in combinatorial physics.

This paper is organized as follows.  In Section 2, we define algebras, coalgebras, and bialgebras 
over a commutative ring $R$. In Section 3, we give the definitions of semirings and semimodules, 
and recall some useful facts and constructions. Section 4 contains the
construction of the tensor product of two semimodules over a commutative
semiring. Using this definition, in Section 5 we apply the defining diagrams of 
algebras, coalgebras, and bialgebras to
categories of semimodules and semimodule homomorphisms.  
We treat automata as representation objects of algebras in Section 6, and then
treat languages as elements of the dual algebra of a coalgebra in Section 7. 
In Section 8, we combine the algebraic and coalgebraic viewpoints, and show how 
to run automata in parallel if they are representation objects of a bialgebra.
We give the adjunction between deterministic automata and algebraic automata in
Section 9, and the proof system in Section 10.

\section{Algebras, Coalgebras, and Bialgebras}
\label{section:acb}
We now define algebras, coalgebras, and bialgebras over a commutative ring $R$. 
This material is completely standard; see \cite{bib:fqgt} or \cite{bib:qgpca} (note that Hopf
algebras and quantum groups are special cases of  bialgebras).
\subsection{Algebras}
\begin{definition}
Let $R$ be a commutative ring.  An {\em $R$-algebra} $(A,\cdot,\eta)$ is a ring
$A$ together with a ring homomorphism $\eta: R \rightarrow A$ such that
$\eta(R)$ is contained in the center of $A$ and $\eta(1_R) = 1_A$.
\end{definition}
\begin{rmk}
The function $\eta$ is called the {\em unit map} and defines an action of $R$
on $A$ via $ra = \eta(r)a$, so $A$ is also an $R$-module.  
\end{rmk}
To define an $R$-algebra diagrammatically, consider $A$ as an $R$-module.  
Multiplication in $A$ is an $R$-bilinear map $A \times A \rightarrow A$, by distributivity 
and the fact that $\eta(R)$ is contained in the center of $A$.  By the universal
property of the tensor product, multiplication defines a unique $R$-linear map $\mu: A \otimes A \rightarrow A$ 
(all tensor products in this section are over $R$).  Associativity of
multiplication implies that the following diagram commutes: 
$$
\xymatrix{
& A \otimes A \otimes A\ar[dl]_{\mu \otimes 1_A} \ar[dr]^{1_A \otimes \mu} & \\
A \otimes A\ar[dr]_\mu & & A \otimes A\ar[dl]^{\mu} \\
& A. & }
$$
The properties of the unit map can be expressed by the following commutative diagram 
(Recall that $A \otimes R \cong A \cong R \otimes A$): 
$$
\xymatrix{
A \ar@/^1.5pc/[rrrr]^{1_A} \ar@{=>}[rr]^{\eta \otimes 1_A}_{1_A \otimes \eta} & 
& A \otimes A \ar[rr]^{\mu}& & A.}
$$
Hence the diagrammatic definition of an $R$-algebra is an $R$-module $A$ together with $R$-module homomorphisms
$\mu:A \otimes A \rightarrow A$ and $\eta: R \rightarrow A$ such that the above diagrams commute.

\begin{example}
\label{example:one}
Let $R$ be a commutative ring and $P$ be the set of polynomials over
noncommuting variables $x,y$ with coefficients in $R$.  Addition and multiplication of polynomials
make $P$ into a ring.  To make $P$ into an $R$-algebra, define $\eta(r)$ to be
the constant polynomial $p(x,y)=r$ for $r \in R$.
\end{example}
\noindent Structure-preserving maps between algebras are called {\em algebra
maps}.
\begin{definition}
Let $A$ and $B$ be $R$-algebras.  An {\em algebra map} is an $R$-linear map 
$f:A \rightarrow B$ such that $f(a_1a_2)
 = f(a_1)f(a_2)$ for all
$a_1,a_2 \in A$, and $f(1_A) = 1_B$.
Equivalently, an $R$-linear map $f$ such that the following diagrams commute:
$$
\xymatrix{
A \otimes A \ar[r]^{f \otimes f} \ar[d]_{\mu_A} & B \otimes B \ar[d]^{\mu_B} \\
A \ar[r]^f & B}
\xymatrix{
& R \ar[dl]_{\eta_A} \ar[dr]^{\eta_B} & \\
A \ar[rr]^f & & B.}
$$
\end{definition}

\noindent Given two $R$-algebras $A$ and $B$, $A \otimes B$ becomes an
$R$-algebra with multiplication 
$$(a \otimes b) \cdot (a' \otimes b') = aa' \otimes bb'.$$ 
Diagrammatically, this multiplication can be expressed as a morphism $$
\xymatrix{
(A \otimes B) \otimes (A \otimes B) \ar[rr]^{\cong}_{1_A \otimes \sigma \otimes 1_B} & &
(A \otimes A) \otimes (B \otimes B) \ar[rr]^{\hspace{8 mm}\mu_A \otimes \mu_B} & &
A \otimes B.}
$$
Here $\sigma: A \otimes B \rightarrow B \otimes A$; $\sigma(a \otimes b) = (b \otimes a)$ is the 
usual transposition map. \\
The unit of $A \otimes B$ is given by
$$
\xymatrix{
R \ar[r]^{\cong \hspace{3.5 mm}} & R \otimes R \ar[rr]^{\eta_A \otimes \eta_B}
& & A \otimes B.} $$

\subsection{Coalgebras}
\begin{definition}
Let $R$ be a commutative ring.  An $R$-{\it coalgebra} $(C,\Delta,\epsilon)$ is an $R$-module $C$ 
together with an $R$-linear coassociative function $\Delta: C \rightarrow C
\otimes C$, called {\it comultiplication}, and an $R$-linear {\it counit} map 
$\epsilon: C \rightarrow R$, which satisfy the diagrams below. 
\end{definition}
\noindent Coassociativity of $\Delta$ means that the following diagram commutes:
$$
\xymatrix{
& C \otimes C \otimes C  & \\
C \otimes C\ar[ur]^{\Delta \otimes 1_C} & & C \otimes C\ar[ul]_{1_C \otimes \Delta} \\
& C. \ar[ul]^\Delta \ar [ur]_\Delta&}
$$
\noindent Diagrammatically, the axioms of the counit map are given by:
$$
\xymatrix{
C \ar[rr]^{\Delta}  \ar@/^1.5pc/[rrrr]^{1_C} & & C \otimes C \ar@{=>}[rr]^{\epsilon \otimes 1_C}_{1_C \otimes \epsilon} 
& & C.}
$$
When performing calculations involving comultiplication, we sometimes use the
expression
$$\Delta(c) = \sum_i c_{(1)} \otimes c_{(2)}$$
to express how $c$ is ``split'' into elements of $C \otimes C$.
\begin{example}
\label{example:two}
Let $P$ the set of polynomials over noncommuting variables $x,y$ with
coefficients in $R$ from Example \ref{example:one}. The map $\Delta: P
\rightarrow P \otimes P$, defined on monomials $w$ by $\Delta(w) = w \otimes w$ 
and extended linearly to all of $P$, is coassociative. Defining the counit map
$\epsilon: P \rightarrow R$ to be evaluation at (1,1) makes $(P, \Delta,
\epsilon)$ into an $R$-coalgebra.
\end{example}
\noindent Coalgebras also have structure-preserving maps.
\begin{definition}
Let $C,D$ be $R$-coalgebras.  A {\it coalgebra map} is an $R$-module homomorphism $g: C \rightarrow D$ such that 
the following diagrams commute:
$$
\xymatrix{
C \otimes C \ar[r]^{g \otimes g} & D \otimes D
\\ C \ar[r]^g \ar[u]^{\Delta_C} & D \ar[u]_{\Delta_D}}
\xymatrix{ & R & \\
C \ar[rr]^g \ar[ur]^{\epsilon_C} & & D \ar[ul]_{\epsilon_D}.}
$$
\end{definition}
\noindent Given $R$-coalgebras $C$ and $D$, there is a natural $R$-coalgebra
structure on $C \otimes D$.  Comultiplication and counit are defined by
$$
\xymatrix{
C \otimes D \ar[rr]^{\Delta_C \otimes \Delta_D \hspace{9 mm}} &
&
(C \otimes C) \otimes (D \otimes D) \ar[rr]^{\cong}_{1_C \otimes \sigma \otimes 1_D} & &
(C \otimes D) \otimes (C \otimes D).}
$$
$$
\xymatrix{
C \otimes D \ar[rr]^{\epsilon_C \otimes \epsilon_D \hspace{3 mm}} & & R \otimes
R \cong R.} $$
\subsection{Bialgebras}
\begin{definition}
Let $R$ be a commutative ring.  An $R$-{\it bialgebra} $(B,\mu,\eta,\Delta,\epsilon)$ is an $R$-module $B$ which
is a both an $R$-algebra and an $R$-coalgebra, which also satisfies:
$$\Delta(ab) = \Delta(a)\Delta(b),~\Delta(1) = 1 \otimes 1,~\epsilon(ab)=\epsilon(a)\epsilon(b),~\epsilon(1)=1.
$$
\end{definition}
\noindent Note that the product $\Delta(a)\Delta(b)$ takes place in the algebra structure on \\
$B \otimes B$.
The defining diagrams for a bialgebra are as follows:
$$
\xymatrix{
B \otimes B \ar[r]^\mu \ar[d]_{\Delta \otimes \Delta} & B \ar[r]^\Delta& B \otimes B \\
B \otimes B \otimes B \otimes B \ar[rr]^{1_B \otimes \sigma \otimes 1_B}&  & B \otimes B \otimes B \otimes B \ar[u]_{
\mu \otimes \mu}}
$$
$$
\xymatrix{
B \otimes B \ar[r]^{\epsilon \otimes \epsilon} \ar[d]_{\mu} & R \otimes R \ar[r]^{\eta \otimes \eta} \ar[d]_{\cong} & 
B \otimes B\\
B \ar[r]^{\epsilon} & R \ar[r]^{\eta} & B \ar[u]_{\Delta}}
\xymatrix{
 & B \ar[dr]^{\epsilon} & \\
R \ar[ur]^{\eta} \ar[rr]^{1_R} & & R.}
$$
\begin{rmk}
The following are equivalent:
\begin{enumerate}
\item $B$ is a bialgebra,
\item $\mu:B \otimes B \rightarrow B$ and $\eta:R \rightarrow B$ are
$R$-coalgebra maps,
\item $\Delta: B \rightarrow B \otimes B$ and $\epsilon: B \rightarrow R$ are
$R$-algebra maps.
\end{enumerate}
\end{rmk}
\noindent Note the ``self-duality'' of the defining diagrams of a bialgebra: swapping $\Delta$ for $\mu$, 
$\epsilon$ for $\eta$, and
reversing the direction of all arrows yields the same diagrams.
\begin{example}
\label{example:three}
The set of polynomials $P$ with the $R$-algebra structure of Example
\ref{example:one} and $R$-coalgebra structure of Example \ref{example:two} 
is an $R$-bialgebra.
\end{example}
\begin{example}
\label{example:four}
More generally, let $M$ be a monoid and $R$ a commutative ring.  Let $R(M)$ be the free $R$-module 
on $M$. Define multiplication in $R(M)$ by extending multiplication in $M$
linearly. Then $R(M)$ is an $R$-algebra with unit map $\eta(r) = r1_M$. There is an 
$R$-coalgebra structure on $R(M)$; define $$\Delta(m) = m \otimes m$$ $$\epsilon(m) = 1$$ for
$m \in M$ and extend linearly to $R(M)$.  A straightforward calculation shows that $R(M)$ is 
an $R$-bialgebra.
\end{example}
\noindent Finally, we give the definition of a bialgebra map.
\begin{definition}
Let $B, B'$ be bialgebras.  An $R$-linear map $f: B \rightarrow B'$ is a {\em
bialgebra map} if $f$ is both an algebra map and a coalgebra map.
\end{definition}
\section{Semirings and Semimodules}
When studying automata and formal languages, it is natural to work over
{\em semirings}, which are ``rings without subtraction''.
\begin{definition}
A {\it semiring} is a structure $(K,+,\cdot,0,1)$ such that $(K,+,0)$ is a 
commutative monoid, $(K,\cdot,1)$ is a monoid,
and the following laws hold:
$$j(k+l) = jk + jl$$
$$(k+l)j = kj + lj$$
$$0k = k0 = 0$$
for all $j,k,l \in K$. If $(K, \cdot, 1)$ is a commutative monoid, then $K$
is said to be a {\em commutative semiring}.  If $(K,+,0)$ is an idempotent
monoid, then $K$ is said to be an {\em idempotent semiring}.  
\end{definition}
The representation objects of semirings are known as {\em semimodules}.
\begin{definition} 
Let $K$ be a semiring.  A {\em left $K$-semimodule} is a commutative monoid
$(M,+,0)$ along with a left action of $K$ on $M$.
The action satisfies the following axioms:
$$(j + k)m = jm + km$$
$$j(m + n) = jm + jn$$
$$(jk)m = j(km)$$
$$1_Km = m$$
$$k0_M = 0_M = 0_Km$$
for all $j,k \in K$ and $m,n \in M$.  If addition in $M$ is
idempotent, $M$ is said to be an {\em idempotent left $K$-semimodule}.
\end{definition}
\noindent Right $K$-semimodules are defined analogously; in the sequel we
give only ``one side'' of a definition.  If $K$ is commutative, then every
left $K$-semimodule can be regarded as a right $K$-semimodule, and vice versa.  
In this case, we omit the words ``left'' and ``right''.  

\begin{example}
Let $K$ be a semiring and $m,n$ be positive integers.  The set of $m \times n$
matrices over $K$ is a left $K$-semimodule, and the set of $m \times m$
matrices over $K$ is a semiring, using the standard definitions of
matrix addition, multiplication, and left scalar multiplication.
\end{example}

Semimodules can be combined using the operations of direct sum and direct product.

\begin{definition}
Let $K$ be a semiring and $\{M_i~|~ i \in I\}$ be a collection of
left $K$-semimodules for some index set $I$.
Let $M$ be the cartesian product of the underlying sets of the $M_i$'s.
The {\it direct product} of the $M_i$'s, denoted $\prod M_i$, is the set $M$ 
endowed with pointwise addition and scalar multiplication. The {\it direct
sum} of the $M_i$'s, denoted $\bigoplus M_i$, is the subsemimodule of 
$\prod M_i$ in which all but finitely many of the coordinates are 0.
\end{definition}
\begin{rmk}
As usual, direct products and direct sums coincide when $I$ is finite.
\end{rmk}
\noindent Homomorphisms, congruence relations, and factor semimodules 
are all defined standardly.
\begin{definition}
Let $K$ be a semiring and $M,N$ be left $K$-semimodules.  A function 
$\phi:M \rightarrow N$ is a {\it left K-semimodule homomorphism} if
$$\phi(m + m') = \phi(m) + \phi(m') {\rm ~for~ all~} m,m' \in M$$
$$\phi(km) = k\phi(m) {\rm ~for~ all~} m \in M, k \in K.$$
Such $\phi$ are also called $K$-{\it linear maps}.
\end{definition}
\begin{definition}
For a given semiring $K$, let \Kmod \text{ } be the category
of left $K$-semimodules and $K$-linear maps.
\end{definition}
\begin{definition}
Let $K$ be a semiring, $M$ a left $K$-semimodule, and $\equiv$ an equivalence 
relation on $M$.  Then $\equiv$ is a 
{\it congruence relation} if and only if 
$$m \equiv m' ~\mbox{and}~ n \equiv n' \mbox{ implies } m + n \equiv m' + n'$$
$$m \equiv m' \mbox{ implies } km \equiv km'$$ 
for all $k \in K$, $m,m',n,n' \in M$.
\end{definition}
\begin{definition}
Let $K$ be a semiring, $M$ a left $K$-semimodule, and $\equiv$ a congruence 
relation on $M$.  For each $m \in M$, let $[m]$ be the
equivalence class of $m$ with respect to $\equiv$.  Let $M/ \equiv$ be 
the set of all such equivalence classes.  Then $M / \equiv$
is a left $K$-semimodule with the following operations:
$$[m] + [n] = [m+n]$$
$$k[m] = [km]$$
for all $m,n \in M, k \in K$.  This semimodule is known as the 
{\it factor semimodule} of $M$ by $\equiv$.
\end{definition}
\begin{definition}
Let $K$ be a semiring and $X$ a nonempty set.  The {\em free left $K$-semimodule
on $X$} is the set of all finite formal sums of the form
$$k_1x_1 + k_2x_2 + \cdot \cdot \cdot + k_nx_n$$
with $k_i \in K$ and $x_i \in X$, i.e., the set of all $f \in K^X$ with finite
support. Addition and the action of $K$ are defined pointwise.  
\end{definition}
Equivalently, one can define a left $K$-semimodule $M$ to be free if
and only if $M$ has a basis \cite{bib:sata}.
\begin{definition}
Let $M$ be a left $K$-semimodule and $X$ a nonempty subset of $M$.
Then there is a $K$-linear map $\phi$ from the left $K$-semimodule of 
all functions $f \in K^X$ with finite support to $M$ given
by
$$\phi(f) = \sum_{x \in X} f(x)x.$$
If $\phi$ is surjective, then $X$ is said to be a 
{\em set of generators} of $M$.  If $\phi$ is injective,
then $X$ is said to be {\em linearly independent}.
If $\phi$ is a bijection, then $X$ is said to be a 
{\em basis} of $M$.
\end{definition}
\begin{rmk}
If $M$ is a left $K$-semimodule with a basis of size $m \in \mathbb{N}$,
and $N$ is a left $K$-semimodule with a basis of size $n \in \mathbb{N}$,
then a $K$-linear map from $M$ to $N$ can be
represented by an $n \times m$ matrix over $K$.
\end{rmk}
In the sequel, we use elementary facts about factor semimodules, 
free semimodules, congruence relations, and homomorphisms without comment. 
See \cite{bib:sata} for proofs.  
\begin{definition}
Let $K$ be a commutative semiring and $M$ a $K$-semimodule.  The
set of all $K$-linear maps $M \rightarrow K$ is denoted \Hom(M,K).
\end{definition}
\begin{rmk}
In the sequel, the notation \Hom(X,Y) always refers to the set
of $K$-linear maps between $X$ and $Y$, considered as $K$-semimodules,
even if $X$ and $Y$ have additional structure.
\end{rmk}
We end this section with two useful lemmas 
concerning dual semimodules.  The proofs are simple generalizations
of the standard proofs for the case when $K$ is a ring.
\begin{lemma}
\label{lemma:hom}
Let $K$ be a commutative semiring and $M$ a $K$-semimodule.  The set
\Hom$(M,K)$ can be endowed with a $K$-semimodule structure.
\end{lemma}
\begin{proof}
\Hom$(M,K)$ is a commutative monoid under
pointwise addition.  Let $f \in$ \Hom$(M,K)$.
The action of $K$ on \Hom$(M,K)$, denoted $\cdot$, is defined 
by $k\cdot(f(m)) = kf(m)$.  Commutativity of $K$ is needed to show that the resulting 
functions are $K$-linear. Since $f$ is $K$-linear, 
$k \cdot f(k'x) = k \cdot k'f(x) = kk'f(x)$.  In order for $k\cdot f$ to be
$K$-linear, we must have $k\cdot f(k'x) = k'k \cdot f(x) = k'kf(x)$.  This
means the equation $kk'f(x) = k'kf(x)$ must hold, which is the case if $K$ is commutative.
\end{proof}
\begin{lemma}
\label{lemma:finitedual}
Let $K$ be a commutative semiring, $X$ be a finite nonempty set, and $F$ the
free $K$-semimodule on $X$.  Then $\Hom(F,K)$ is also a free $K$-semimodule
on a set of size $|X|$.  
\end{lemma}
\begin{proof}
Let $x_1,x_2,\ldots,x_n$ be a basis of $F$ and $f_i \in \Hom(F,K)$
be such that $f_i(x_j) = 1$ if  $i = j$ and 0 otherwise. 
We claim that the $f_i$'s are a basis
of $\Hom(F,K)$. Let $g \in \Hom(F,K)$ and $a_i = g(x_i)$.
The $f_i$'s form a generating set because 
$$g(k_1x_1 + k_2x_2 + \cdots +  k_nx_n) = k_1g(x_1) + k_2g(x_2) + \cdots +
k_ng(x_n),$$ 
and so $g = a_1f_1 + a_2f_2 + \cdots + a_nf_n$.
Moreover, the $f_i$'s are linearly independent; if
$$j_1f_1 + j_2f_2 + \cdots + j_nf_n = j_1'f_1 + j_2'f_2 + \cdots + j_n'f_n,$$ 
then evaluating each side on $x_i$ yields $j_i = j_i'$.
\end{proof}

\section{Tensor Products over Commutative Semirings}
We wish to apply the defining diagrams of algebras, coalgebras, and bialgebras to 
categories of $K$-semimodules and $K$-linear maps.  To do this, we need a 
notion of the tensor product of $K$-semimodules. 
Unfortunately, the literature contains multiple inequivalent definitions of the 
tensor product of $K$-semimodules: the tensor  product as defined in
\cite{bib:sata} is not the same as the tensor product defined in \cite{bib:tpis} 
or \cite{bib:tpie}.  In fact, the tensor product defined in \cite{bib:sata} is the
trivial $K$-semimodule when applied to idempotent $K$-semimodules.

We proceed by assuming that $K$ is commutative and mimicking the construction of 
the tensor product of modules over a 
commutative ring in \cite{bib:lang}.  This is essentially the construction used
in \cite{bib:tpis} and \cite{bib:tpie}.  The point is to work in the appropriate 
category and construct an object with the appropriate universal property.

We recall the universal property of the tensor product over a commutative ring $R$. 
Let $M_1,M_2,...,M_n$ be $R$-modules.  Let $\mathcal{C}$ be the category whose objects 
are $n$-multilinear maps 
$$f: M_1 \times M_2 \times \cdot \cdot \cdot \times M_n \rightarrow F$$
where $F$ ranges over all $R$-modules.  To define the morphisms of $\mathcal{C}$, let 
$$f: M_1 \times M_2 \times \cdot \cdot \cdot \times M_n \rightarrow F \text{~and~}
g:M_1 \times M_2 \times \cdot \cdot \cdot \times M_n \rightarrow G$$
be objects of $\mathcal{C}$.  A morphism $f \rightarrow g$ is an $R$-linear map
$h: F \rightarrow G$ such that $h \circ f = g$. A {\it tensor product} of
$M_1,M_2,...,M_n$, denoted $M_1 \otimes_R M_2 \otimes_R \cdot \cdot \cdot \otimes_R M_n$,
is an initial object in this category. 
When it is clear from context, we omit the subscript on the $\otimes$ symbol. 
By a standard argument, the tensor product is unique up to isomorphism.

We now construct the tensor product of semimodules over a commutative
semiring.  Let $K$ be a commutative semiring and $M_1,M_2,...,M_n$ be
$K$-semimodules. Let $T$ be the free $K$-semimodule on the (underlying) set \\   
$M_1 \times M_2 \times \cdot \cdot \cdot \times M_n$. Let $\equiv$ be the
congruence relation on $T$ generated by the equivalences $$(m_1,...,m_i +_{M_i}
m_i',...,m_n) \equiv (m_1,...,m_i,...,m_n) +_T (m_1,...,m_i',...,m_n)$$ 
$$(m_1,...,km_i,...,m_n) \equiv k(m_1,...,m_i,...,m_n)$$
for all $k \in K, m_i,m_i' \in M_i, 1 \leq i \leq n$.

Let $i:M_1 \times M_2 \times \cdot \cdot \cdot \times M_n \rightarrow T$ be the canonical injection
of $M_1 \times M_2 \times \cdot \cdot \cdot \times M_n$ into $T$. 
Let $\phi$ be the composition of $i$ and the quotient map $q:T \rightarrow T / \equiv$.  
\begin{lemma}
The map $\phi$ is multilinear and is a tensor product of \\
$M_1, M_2,...,M_n$.
\end{lemma}
\begin{proof}
Multilinearity of $\phi$ is obvious from its definition.  Let $G$ be a $K$-semimodule and 
$$g:M_1 \times M_2 \times \cdot \cdot \cdot \times M_n \rightarrow G$$
be a $K$-multilinear map.  By freeness of $T$, there is an induced $K$-linear map 
$\gamma: T \rightarrow G$ such that the following diagram commutes:
$$
\xymatrix{
 & T \ar[dd]^{\gamma} \\
 M_1 \times M_2 \times \cdot \cdot \cdot \times M_n \ar[ur]^i \ar[dr]_g & \\
 & G.}
$$
The homomorphism $\gamma$ defines a congruence relation, denoted $\equiv_{\gamma}$, on $T$ via
$$t \equiv_{\gamma} t' \text{ if and only if } \gamma(t) = \gamma(t')$$
for all $t,t' \in T$.  Since $g$ is $K$-multilinear, we have $\equiv~ \subseteq ~ \equiv_{\gamma}$, 
where $\equiv$ is the congruence relation used in the definition of the tensor product.  Therefore 
$\gamma$ can be factored through $T / \equiv$, and there is a $K$-linear map
$$g_*: T / \equiv \rightarrow G$$ 
making the following diagram commute: 
$$
\xymatrix{
 & T / \equiv \ar[dd]^{g_*} \\
M_1 \times M_2 \times \cdot \cdot \cdot \times M_n \ar[ur]^{\phi} \ar[dr]_g \\
 & G.}
$$
The image of $\phi$ generates $T / \equiv$, so $g_*$ is uniquely determined.
\end{proof}
For $x_i \in M_i$, we denote $\phi(x_1,x_2,...,x_n)$ by $x_1 \otimes x_2 \otimes \cdot \cdot \cdot \otimes x_n$. Tensor 
products enjoy many useful properties.

\begin{lemma}
\label{lemma:tensorproperties}
Let $K$ be a commutative semiring and $N, M_1,M_2,...,M_n$ be $K$-semimodules.  Then:
\begin{enumerate}
\item There is a unique isomorphism
$$(M_1 \otimes M_2) \otimes M_3 \rightarrow M_1 \otimes (M_2 \otimes M_3)$$
such that 
$(m_1 \otimes m_2) \otimes m_3 \mapsto m_1 \otimes (m_2 \otimes m_3)$ for all
$m_i \in M_i$.
\item There is a unique isomorphism $M_1 \otimes M_2 \rightarrow M_2 \otimes M_1$ such that \\
$m_1 \otimes m_2 \mapsto m_2 \otimes m_1$
for all $m_i \in M_i$.
\item $K \otimes M_1 \cong M_1$
\item Let $\phi:M_1 \rightarrow M_3$ and $\psi:M_2 \rightarrow M_4$ be $K$-linear maps.  There is a 
unique $K$-linear map $\phi \otimes \psi: M_1 \otimes M_2 \rightarrow M_3
\otimes M_4$ such that \\ $(\phi \otimes \psi)(m_1 \otimes m_3) = \phi(m_1) \otimes
\psi(m_2)$ for all $m_1 \in M_1, m_2 \in M_2$.
\item $N \otimes \bigoplus_{i \in I} M_i \cong \bigoplus_{i \in I} N \otimes
M_i$ for any index set $I$.
\item Let $M$,$N$ be free $K$-semimodules, with bases $\{m_i\}_{i \in I}$ and
$\{n_j\}_{j \in J}$, respectively.  Then $M \otimes N$ is a free $K$-semimodule
with basis $\{m_i \otimes n_j\}$.
\end{enumerate}
\end{lemma}
\begin{proof}
In \cite{bib:lang}, these properties are proven for tensor products over commutative rings. 
The proofs rely on the universal property of the tensor product and are also
valid in this case.
\end{proof}
\section{$K$-algebras, $K$-coalgebras, and $K$-bialgebras}
Let $K$ be a commutative semiring.  We define $K$-algebras, $K$-coalgebras, 
$K$-bialgebras, and their respective maps by applying the relevant diagrams
from Section \ref{section:acb} to the category of $K$-semimodules and $K$-linear
maps.  To avoid clumsy terminology, we do not use the terms ``semi-algebra'',
``semi-coalgebra'', or ``semi-bialgebra''.
\begin{example}
\label{example:classical}
Let $\Sigma = \{x,y\}$ be a set of noncommuting variables.  Let $P$ be the set of polynomials 
over $\Sigma$ with coefficients from the two-element idempotent semiring $K$.
Multiplication of polynomials is readily seen to be a $K$-bilinear function $P
\times P \rightarrow P$, and therefore corresponds to a $K$-linear map $P \otimes_K P \rightarrow P$. Moreover, this map 
satisfies the associativity diagram.
The underlying $K$-semimodule of $P$ is the free $K$-semimodule on the set of
all words $w$ over $\{x,y\}$, so $P \otimes P$ is the free $K$-semimodule with
basis $\{w \otimes w'\}$ by Lemma \ref{lemma:tensorproperties}.6.  The $K$-linear map
$\eta:K \rightarrow P$ such that $\eta(k) \mapsto \lambda xy.k$ satisfies the defining diagram of the 
unit map, and so $P$ together with these maps forms a $K$-algebra.  

The $K$-linear map $\Delta$ defined on monomials as $\Delta(w) = w \otimes w$
and extended linearly to all of $P$ is easily seen to be coassociative.  Defining $\epsilon(p(x,y)) = p(1,1)$ makes
$P$ into a $K$-coalgebra.  Furthermore, these maps satisfy the compatibility condition of a $K$-bialgebra, 
so $P$ is a $K$-bialgebra.  

We refer to constructions involving $P$ as ``the classical case'' throughout
the sequel.
\end{example}
\begin{example}  Given any set $X$ and commutative semiring $K$,
it follows from general considerations that there is a free $K$-algebra
on $X$, which we denote $KX^*$, and furthermore that there is an adjunction
between the category of $K$-algebras and $K$-algebra maps and {\bf Set}.  
\end{example}
One can associate two $K$-algebras to any $K$-semimodule $M$.
\begin{lemma}
Let $M$ be a $K$-semimodule over a commutative semiring $K$.  
The set of {\em left endomorphisms} of $M$, denoted 
$\Endl(M)$, is the set of all $K$-linear maps $M \rightarrow M$ endowed
with the following operations.
Addition and scalar multiplication are defined pointwise.  Let
$f,g$ be $K$-linear maps $M \rightarrow M$.  Define
$$fg(a) = f(g(a)).$$
Similarly, let $\Endr(M)$ be the set of all $K$-linear maps
$M \rightarrow M$ endowed with pointwise addition and scalar multiplication,
and define multiplication by
$$(a)fg = ((a)f)g.$$
Then $\Endl(M)$ and $\Endr(M)$ are $K$-algebras.
\end{lemma}
\begin{proof}
Calculation.
\end{proof}
\begin{rmk}
The distinction between $\Endl(M)$ and $\Endr(M)$ allows us to define
automata which read input words from right to left, and automata which read
input words from left to right.
\end{rmk}
\section{$K$-algebras and Automata}
In Example \ref{example:classical}, we defined a $K$-algebra on the
set of polynomials over the noncommuting variables $\{x,y\}$.  We can also
think of elements of this algebra as finite sums of words over the alphabet
$\{x,y\}$. In this section, we generalize this idea and use the actions of
$K$-algebras on $K$-semimodules to define transitions of automata, and list 
several analogs between algebraic constructions and constructions on automata.
\begin{definition}
Let $A$ be a $K$-algebra and $M$ be a $K$-semimodule.  A {\em left action} 
of $A$ on $M$ is a $K$-linear map $A \otimes M \rightarrow M$, denoted $\triangleright$, 
satisfying $$(aa') \triangleright m = a \triangleright (a' \triangleright m)$$
$$1 \triangleright m = m$$ for all $a,a' \in A, m \in M$.
\end{definition}
\noindent Right actions are defined analogously as $K$-linear maps $\triangleleft: M \otimes 
A \rightarrow M$.  To define an automaton, we also need a start state and an
observation function.
\begin{definition}
A {\em left $K$-linear automaton} $\mathcal{A} = (M,A,s,\triangleright, \Omega)$ 
consists of the following:
\begin{enumerate}
\item A $K$-algebra $A$, a $K$-semimodule $M$, and a left action
$\triangleright$ of $A$ on $M$,
\item An element $s \in M$, called the {\em start vector},
\item A $K$-linear map $\Omega: M \rightarrow K$, called the {\em observation
function}.
\end{enumerate}
\end{definition}
\begin{rmk}
Equivalently, we could have defined a $K$-linear {\em start function}
$$\alpha: K \rightarrow M$$
and set $s = \alpha(1)$.  This is useful in Section
\ref{section:determinize} below, but can add unnecessary symbols to
proofs.  We use both variants, depending on the situation.
\end{rmk}
Automata are ``pointed observable representation
objects'' of a $K$-algebra $A$. Right automata are defined similarly using a right action $\triangleleft$. 
In the sequel, we give only ``one side'' of a theorem or definition
involving automata; the other follows mutatis mutandis.  Intuitively, right automata read inputs
from left to right, and left automata read inputs from right to left (see
Example \ref{example:reversal} below).
\begin{example}
\label{example:algebraicautomaton}
Consider the following classical automaton:
$$
 \xymatrix{\ar[r]
& *++[o][F]{s_1}\ar[r]_{x} \ar@(ur,ul)^{x}
& *++[o][F=]{s_2}\ar@/_1.5pc/[l]^{y}
& }
$$
We provide a translation of this automaton into the framework of $K$-algebra
representations.

Let $K$ be the two-element idempotent semiring.  Let $M$ be the free
$K$-semimodule on the set $\{s_1,s_2\}$, and let $P$ be defined as in Example 
\ref{example:classical}.  Define a right action of the generators of $P$ 
(as a $K$-algebra) on $M$ as follows: $$
\left[
\begin{array}{lr}
k_1 & k_2 \\ \end{array}
\right] 
\triangleleft x =
\left[
\begin{array}{lr}
k_1 & k_2 \\ \end{array}
\right]
\left[
\begin{array}{lr}
1 & 1 \\
0 & 0 \\ \end{array}
\right]    
$$
$$
\left[
\begin{array}{lr}
k_1 & k_2 \\ \end{array}
\right] 
\triangleleft y =
\left[
\begin{array}{lr}
k_1 & k_2 \\ \end{array}
\right]
\left[
\begin{array}{lr}
0 & 0 \\
1 & 0 \\ \end{array}
\right]    
$$
and extend algebraically to an action of $P$ on $M$.
The start vector is
$$
\left[
\begin{array}{lr}
1 & 0 \\ \end{array}
\right]
$$
and the observation function is
$$
\Omega
\left(
\left[
\begin{array}{lr}
k_1 & k_2 \\ \end{array}
\right]
\right) = 
\left[
\begin{array}{lr}
k_1 & k_2 \\ \end{array}
\right]
\left[
\begin{array}{lr}
0 \\
1 \\ \end{array}
\right].
$$
\end{example}
Automata determine elements of $\Hom(A,K)$, as in \cite{bib:bar}.

\begin{definition}
Let $\mathcal{A} = (M,A,s,\triangleright, \Omega)$ be a left
$K$-linear automaton. The {\em language accepted} by $\mathcal{A}$ is the
function $\rho_{\mathcal{A}}: A \rightarrow K$ such that 
$$\rho_{\mathcal{A}}(a) = \Omega(a \triangleright s).$$
\end{definition}
\begin{lemma}
The function $\rho_{\mathcal{A}}$ is an element of $\Hom(A,K)$.
\end{lemma}
\begin{proof}
Immediate since $\triangleright$ and $\Omega$ are $K$-linear maps.
\end{proof}
\begin{definition}
Let $\mathcal{A}$ and $\mathcal{B}$ be left $K$-linear automata.
If $\rho_{\mathcal{A}} = \rho_{\mathcal{B}}$, then $\mathcal{A}$
and $\mathcal{B}$ are said to be {\em equivalent}.
\end{definition}
\noindent Functions between automata which preserve the language 
accepted are central to the theory of automata; such functions have
$K$-algebraic analogs.
\begin{definition}
Let $\mathcal{A} = (M,A,s_{\mathcal{A}},\triangleright_{\mathcal{A}},
\Omega_{\mathcal{A}})$ and ${\mathcal{B}} =
(N,A,s_{\mathcal{B}},\triangleright_{\mathcal{B}}, \Omega_{\mathcal{B}})$ be
left $K$-linear automata. An {\em $K$-linear automaton
morphism} from $\mathcal{A}$ to $\mathcal{B}$ is a map $\phi: M \rightarrow N$ such that
\begin{equation}
\phi(s_{\mathcal{A}}) = s_{\mathcal{B}}
\label{eqno1}
\end{equation}
\begin{equation}
\phi(a \triangleright_{\mathcal{A}} m) = a \triangleright_{\mathcal{B}} \phi(m)
\label{eqno2}
\end{equation}
\begin{equation}
\Omega_{\mathcal{A}}(m) = \Omega_{\mathcal{B}}(\phi(m))
\label{eqno3}
\end{equation}
\end{definition}
\noindent for all $m \in M$ and $a \in A$.
\begin{rmk}
Let $V$ and $W$ be $R$-modules.  In the theory of $R$-algebras, an $R$-linear
map $f: V \rightarrow W$ which satisfies (\ref{eqno2}) is known as a {\em linear intertwiner}.
\end{rmk}
\begin{rmk} In the theory of automata, functions formally similar to automaton morphisms have been called
{\em linear sequential morphisms} \cite{bib:aac}, {\em relational simulations}
\cite{bib:brwa}, {\em boolean bisimulations} \cite{bib:bbv}, and {\it
disimulations} \cite{bib:apgka}. Disimulations are based on the {\em
bisimulation lemma} of Kleene algebra \cite{bib:ko94}.
\end{rmk}
The following theorem, or a minor variant, is proven in most of the references
mentioned in the above remark.  
\begin{theorem}
\label{theorem:samelanguage}
Let $\mathcal{A} = (M,A,s_{\mathcal{A}},\triangleright_{\mathcal{A}},
\Omega_{\mathcal{A}})$ and $\mathcal{B} =
(N,A,s_{\mathcal{B}},\triangleright_{\mathcal{B}}, \Omega_{\mathcal{B}})$ be
left $K$-linear automata, and let $\phi:\mathcal{A} \rightarrow
\mathcal{B}$ be a $K$-linear automaton morphism.  Then 
$\mathcal{A}$ and $\mathcal{B}$ are equivalent.
\end{theorem}
\begin{proof}
For any $a \in A$,
\begin{align*}
\Omega_{\mathcal{A}}(a \triangleright_{\mathcal{A}} s_{\mathcal{A}}) 
&=\Omega_{\mathcal{B}}(\phi(a \triangleright_{\mathcal{A}} s_{\mathcal{A}})) \\ 
&= \Omega_{\mathcal{B}}(a\triangleright_{\mathcal{B}} \phi(s_{\mathcal{A}})) \\ 
&= \Omega_{\mathcal{B}}(a \triangleright_{\mathcal{B}} s_{\mathcal{B}}).
\end{align*}
\end{proof}
\noindent A simple calculation proves the following lemma.
\begin{lemma}
Let $\mathcal{A},\mathcal{B},\mathcal{C}$ be left $K$-linear automata and
$\phi:\mathcal{A} \rightarrow \mathcal{B}$, $\phi':\mathcal{B} \rightarrow
\mathcal{C}$ be automaton morphisms. Then $\phi' \circ \phi: \mathcal{A}
\rightarrow \mathcal{C}$ is an automaton morphism.
\end{lemma}
Furthermore, for a left $K$-linear automaton $\mathcal{A}$, the identity map of
the underlying $K$-semimodule of $\mathcal{A}$ is an automaton morphism.  We
therefore have the following.
\begin{lemma}
For a given commutative semiring $K$, the collection of $K$-linear automata and
automaton morphisms forms a category.
\end{lemma}
Let $A$ be a $K$-algebra.  Elements of $\Hom(A,K)$ can be added 
and scaled by $K$, since $\Hom(A,K)$ is a $K$-semimodule by Lemma
\ref{lemma:hom}. Given automata $\mathcal{A}$ and $\mathcal{B}$, there is an automaton accepting
$\rho_{\mathcal{A}} + \rho_{\mathcal{B}}$, and given $k \in K$, there is an
automaton accepting $k\rho_{\mathcal{A}}$.
\begin{definition}
Let $\mathcal{A} =
(M,A,s_{\mathcal{A}},\triangleright_{\mathcal{A}},\Omega_{\mathcal{A}})$ and 
$\mathcal{B} = (N,A,s_{\mathcal{B}},\triangleright_{\mathcal{B}},
\Omega_{\mathcal{B}})$ be left $K$-linear automata. The {\em direct
sum} of $\mathcal{A}$ and $\mathcal{B}$ is the left $K$-linear automaton 
$\mathcal{A} \oplus \mathcal{B} = (M \oplus N, A,
(s_{\mathcal{A}},s_{\mathcal{B}}), \triangleright_{\mathcal{A} \oplus
\mathcal{B}}, \Omega_{\mathcal{A}} \oplus \Omega_{\mathcal{B}})$, where
$$\triangleright_{\mathcal{A} \oplus \mathcal{B}}: A \otimes(M \oplus N) 
\rightarrow M \oplus N,$$ 
$$\triangleright_{\mathcal{A} \oplus \mathcal{B}}(a \otimes (m,n)) = 
((a \triangleright_{\mathcal{A}} m),(a \triangleright_{\mathcal{B}} n))$$ and
$$\Omega_{\mathcal{A} \oplus \mathcal{B}}: M \oplus N \rightarrow K,$$
$$\Omega_{\mathcal{A} \oplus \mathcal{B}}(m,n) =
\Omega_{\mathcal{A}}(m)+\Omega_{\mathcal{B}}(n).$$
\end{definition}
\noindent The verification that $\triangleright_{\mathcal{A} \oplus \mathcal{B}}$ is 
an action of $A$ on $M \oplus N$ is straightforward.
\begin{theorem}
Let $\mathcal{A}= (M,A,s_{\mathcal{A}},\triangleright_{\mathcal{A}},
\Omega_{\mathcal{A}})$ and $(N,A,s_{\mathcal{B}},\triangleright_{\mathcal{B}},
\Omega_{\mathcal{B}})$ be left $K$-linear automata. Then
$\rho_{\mathcal{A} \oplus \mathcal{B}}(a) = \rho_{\mathcal{A}}(a) +
\rho_{\mathcal{B}}(a)$ for all $a \in A$.
\end{theorem}  
\begin{proof}
For any $a \in A$,
\begin{align*}
\rho_{\mathcal{A} \oplus \mathcal{B}}(a) &= \Omega_{\mathcal{A} \oplus
\mathcal{B}}(a \triangleright_{\mathcal{A} \oplus \mathcal{B}}
(s_{\mathcal{A}},s_{\mathcal{B}})) \\
&= \Omega_{\mathcal{A} \oplus \mathcal{B}} (a \triangleright_{\mathcal{A}}
s_{\mathcal{A}}, a \triangleright_{\mathcal{B}} s_{\mathcal{B}})\\ 
&=\Omega_{\mathcal{A}}(a \triangleright_{\mathcal{A}}
(s_{\mathcal{A}})) + \Omega_{\mathcal{B}}(a \triangleright_{\mathcal{B}}
(s_{\mathcal{B}})) \\ 
&=\rho_{\mathcal{A}}(a) + \rho_{\mathcal{B}}(a).
\end{align*}
\end{proof}
\begin{theorem}
Let $\mathcal{A}= (M,A,s,\triangleright, \Omega)$ be a left $K$-linear
automaton, and let $k \in K$.  Then $k\rho_{\mathcal{A}} = \rho_{\mathcal{A}'}$, 
where $\mathcal{A}'= (M,A,ks,\triangleright,\Omega)$.
\end{theorem}
\begin{proof}
For any $a \in A$, $\rho_{\mathcal{A}'} = \Omega(a \triangleright ks) = k
\Omega(a \triangleright s) = k\rho_{\mathcal{A}}$ by linearity. 
\end{proof}
Algebra maps can be used to translate the input of an automaton.
\begin{definition}
Let $A,A'$ be $K$-algebras and $f: A \rightarrow A'$ a $K$-algebra map.  Suppose $A'$ acts on a
$K$-semimodule $M$.  Then $A$ also acts on $M$ according to the formula
$$a \triangleright m = f(a) \triangleright m$$
for $a \in A, m \in M.$ This is known as the {\it pullback} of the action of $A'$.
\end{definition}
\noindent Automata theorists will recognize pullbacks as the main ingredient in the proof 
that regular languages are closed under inverse homomorphisms.

Finally, we provide an example in which we reverse certain $K$-linear automata
using dual $K$-semimodules.
\begin{example}
\label{example:reversal}
Let $\mathcal{A} = (M,A,s,\triangleleft, \Omega)$ be a right $K$-linear
automaton, and suppose that $M$ is a free $K$-semimodule on a finite set $X$ and $A$ is the
free $K$-algebra on a finite $\Sigma$.  Then the left $K$-linear automaton
$\mathcal{B} =(\Hom(M,K),A, \Omega, \triangleright, \alpha^*)$, where
$$
a \triangleright f(m) = f(m \triangleleft a)
$$ 
and 
$$
\alpha^*(m) =  m \cdot s^{\rm T}
$$ 
satisfies
$$
\rho_{\mathcal{A}}(w) = \rho_{\mathcal{B}}(w^{\rm R})
$$ 
for all $w \in \Sigma^*$, where $w^{\rm R}$ is the reverse of a word $w$.
That $A \triangleright \Hom(M,K)$ is an action is an application of the
standard fact that actions on (semi)modules ``change sides'' when the 
modules are dualized.  See, for example, \cite{bib:racom}. 

To prove the claim, let $w = x_1 x_2 \cdots x_n$ with $x_i \in
\Sigma$. For some $k \in K, \rho_{\mathcal{A}}(w)= k.$  Since $M$ is a free
$K$-module, the action of each $x \in \Sigma$ on $M$ is given by right multiplication by a 
$|X| \times |X|$ matrix $M_x$ over $K$, and $\Omega(m) = m \cdot v$ for some
$|X| \times 1$ matrix $v$.
By definition,
$$\Omega(s \triangleright x_1 x_2 \cdots x_n) = s\cdot M_{x_1}M_{x_2}\cdots
M_{x_n} \cdot v = k.$$
Taking the transpose of both sides of this equation yields
$\rho_{\mathcal{B}}(w^{\rm R}) = k^{\rm T} = k$, with the slight
abuse of notation $v^{\rm T} = \Omega$.  Note that the
familiar transpose law from linear algebra, $(AB)^{\rm T} = B^{\rm T}A^{\rm T}$,
is valid for matrices over a commutative semiring.
\end{example}
\section{$K$-coalgebras and Formal Languages}
\label{section:languages}
Let $C$ be a $K$-coalgebra.  By Lemma \ref{lemma:hom}, $\Hom(C,K)$ is a
$K$-semimodule under the operations of pointwise addition and scalar multiplication.  
It is a standard fact that the coalgebra structure of $C$ defines an algebra
structure on $\Hom(C,K)$. 
\begin{definition} Let $(C,\Delta,\epsilon)$ be a $K$-coalgebra and $f,g \in
\Hom(C,K)$.  The {\it convolution product} of $f$ and $g$, denoted $f*g$,
is the element of $\Hom(C,K)$ defined by 
$$f*g = \mu_K \circ (f \otimes g) \circ \Delta.$$
Here $\mu_K$ denotes multiplication in $K$.
\end{definition}
\begin{lemma}
Let $(C,\Delta,\epsilon)$ be a $K$-coalgebra.  There is a $K$-algebra
structure on $\Hom(C,K)$ with multiplication given
by the convolution product and unit $$\eta: K \rightarrow C$$
$$\eta(k) = k \epsilon.$$
In particular, the multiplicative identity is $\epsilon.$
\end{lemma}
\begin{proof}
The operation $*$ is associative because $\Delta$ is coassociative:
$$f * (g * h) = \mu_K(f \otimes (\mu_K (g \otimes h))) \circ ((1 \otimes \Delta)
\circ \Delta)$$
$$ (f * g) * h = \mu_K((\mu_K (f \otimes g)) \otimes h) \circ ((\Delta \otimes
1) \circ \Delta)$$
and coassociativity of $\Delta$ is exactly $((1 \otimes \Delta)
\circ \Delta) = ((\Delta \otimes 1) \circ \Delta)$.  The
rest of the $K$-algebra requirements follow immediately from the definitions.
\end{proof}
The relation between $K$-coalgebras and formal languages is as follows.
Let $P$ be as in Example \ref{example:classical}. Note that an element of
$\Hom(P,K)$ is completely determined by its values on monomials, which we view
as words over $\{x,y\}$. Thus there is a one-to-one correspondence between 
subsets of $\{x,y\}^*$ and elements of $\Hom(P,K)$.

Consider the following comultiplications on $P$, defined on monomials 
and extended linearly:
$$\Delta_1(w) = w \otimes w$$
$$\Delta_2(w) = \sum_{w_1w_2=w} w_1 \otimes w_2.$$
Also consider the comultiplication defined as 
$$\Delta_3(x) = 1\otimes x + x \otimes 1$$
$$\Delta_3(y) = 1 \otimes y + y \otimes 1$$
extended as an algebra map to all of $P$. 
Moreover, we have two $K$-linear maps given by:
$$\epsilon_1(p) = p(1,1)$$
$$\epsilon_2(p) = p(0,0)$$
for all $p \in P$.  Then $(P,\Delta_1,\epsilon_1)$ is a $K$-coalgebra 
(cf. Example \ref{example:two}) as are $(P,\Delta_2,\epsilon_2)$ and
$(P,\Delta_3,\epsilon_2)$.

A simple verification shows that the $K$-algebra on $\Hom(P,K)$ determined
by the $K$-coalgebra $(P,\Delta_1,\epsilon_1)$ corresponds to language
intersection, with the multiplicative identity corresponding to the
language denoted by $(x + y)^*$.  The $K$-coalgebra $(P,\Delta_2,\epsilon_2)$
corresponds to language concatenation with identity $\{\lambda\}$, where
$\lambda$ is the empty word.  Finally, the $K$-coalgebra 
$(P,\Delta_3,\epsilon_2)$ corresponds to the shuffle product of languages, 
again with identity $\{\lambda\}$ (see \cite{bib:ddlawm} and also 
\cite{bib:fqgt}, Proposition 5.1.4). In each case, addition in the $K$-algebra
on $\Hom(P,K)$ corresponds to the union of two languages.

We conclude this section with an example calculation.  Let $f \in \Hom(P,K)$
correspond to the language denoted by $x^*$, and let $g \in \Hom(P,K)$ correspond
to the language denoted by $y^*$.  The following shows that $yx \in f*g$, where
the comultiplication is $\Delta_3$:
\begin{align*}
&\mu_k \circ f \otimes g \circ \Delta_3(xy) = \mu_k \circ f \otimes g (1
\otimes xy + y \otimes x + x \otimes y + xy \otimes 1)\\
&= \mu_K(f(1)\otimes g(xy) + f(y) \otimes g(x) + f(x) \otimes g(y) + f(xy)
\otimes g(1))\\
&= \mu_K(1\otimes 0 + 0 \otimes 0 + 1 \otimes 1 + 0 \otimes 1)\\
&=0 + 0 + 1 + 0\\
&=1.
\end{align*}

\section{Automata, Languages, and $K$-bialgebras }

A $K$-algebra $A$ allows us to define automata which take elements of $A$ as
input. These automata compute elements of $\Hom(A,K)$.  Moreover, a $K$-coalgebra
structure on $A$ defines a multiplication on $\Hom(A,K)$.  We now discuss the
relation between these products on $\Hom(A,K)$ and automata.

We first treat the case in which $A$ is both a $K$-algebra and a $K$-coalgebra, without assuming that $A$ is a $K$-bialgebra.
Let $\mathcal{A} = (M,A,s_{\mathcal{A}},\triangleright_{\mathcal{A}},
\Omega_{\mathcal{A}})$ and 
$\mathcal{B} = (N,A,s_{\mathcal{B}},\triangleright_{\mathcal{B}},\Omega_{\mathcal{B}})$ be 
$K$-linear automata.  Applying the convolution product to $\rho_{\mathcal{A}}$
and $\rho_{\mathcal{B}}$ yields $$\rho_{\mathcal{A}} * \rho_{\mathcal{B}} (a) = 
\mu_K \circ(\sum_i \rho_{\mathcal{A}}(a_{(1)} \triangleright s_{\mathcal{A}})
 \otimes \rho_{\mathcal{B}}(a_{(2)} \triangleright s_{\mathcal{B}})).$$

In words, the convolution product determines a 
formula with comultiplication as a parameter.  Different choices of comultiplication yield different 
products of languages, as discussed in Section \ref{section:languages}.
When the languages are given by automata, we can use this formula to obtain a succinct expression for the 
product of the two languages.

Of course, it would be even better if we could get an automaton accepting
the product of the two languages.  For a $K$-bialgebra, there is an easy way to
construct such an automaton, which relies on a construction from the theory of
bialgebras.

We emphasize that a bialgebra structure is not necessary for an automaton accepting 
$\rho_{\mathcal{A}} * \rho_{\mathcal{B}}$ to 
exist.  Consider $\Delta_2$ and $\Delta_3$ as defined in Section
\ref{section:languages}. They agree on $x$ and $y$, which generate $P$ as an
algebra, so at most one of them can be an algebra map; $\Delta_3$ is an algebra map by definition.  
Therefore $\Delta_2$ is not part of a bialgebra, and so we cannot use the construction to get an automaton 
accepting the concatenation of two languages.  Such an automaton exists, of
course, but it is not given by this construction.

Suppose $B$ is a $K$-bialgebra. The first step is to define an action of $B$
on $M \otimes N$ from actions $B \triangleright_M M$ and $B \triangleright_N N$
(by an action of $B$ on $M$, we mean an action of the underlying algebra of $B$
on $M$). 
\begin{lemma}
Let $B$ be a $K$-bialgebra which acts on $K$-semimodules $M$ and $N$.  Then $B$ acts on $M \otimes N$
according to the diagram
$$
\xymatrix{
B \otimes M \otimes N \ar[r]^{\Delta \otimes 1 \quad}  & B \otimes B \otimes M
\otimes N \ar[r]^{1 \otimes \sigma \otimes 1} & B \otimes M \otimes B \otimes N
\ar[r]^{\quad \quad \triangleright_M \otimes \triangleright_N} & M \otimes N.} 
$$
\end{lemma}
\begin{proof}
It is easy to see that the action of $B$ on $M \otimes N$ is a $K$-linear map
such that $1 \triangleright m \otimes n = m \otimes n$.  To see that
$ab \triangleright m \otimes n = a \triangleright(b \triangleright m \otimes
n)$, note that the equational definition of the action is
$$b \triangleright_{M \otimes N} ( m \otimes n) = \sum_i b_{(1)}\triangleright_M m \otimes b_{(2)} \triangleright_N n.$$
We have
\begin{align*}
ab \triangleright m \otimes n &= \sum_i ab_{(1)}\triangleright_M m \otimes
ab_{(2)} \triangleright_N n\\
&= \sum_i a_{(1)}b_{(1)}\triangleright_M m \otimes
a_{(2)}b_{(2)} \triangleright_N n\\ 
&= a \triangleright(b \triangleright m
\otimes n).
\end{align*}
\end{proof}
\begin{definition}
Let $\mathcal{A} = (M,B,s_{\mathcal{A}},\triangleright_{\mathcal{A}},
\Omega_{\mathcal{A}})$ and 
$\mathcal{B}=(N,B,s_{\mathcal{B}},\triangleright_{\mathcal{B}},\Omega_{\mathcal{B}})$ 
be left $K$-linear automata. The {\em tensor product} of $\mathcal{A}$
and $\mathcal{B}$, denoted $\mathcal{A} \otimes \mathcal{B}$, is the automaton
$(M \otimes N,B,s_{\mathcal{A}} \otimes s_{\mathcal{B}},  \triangleright_{M \otimes N}, 
\Omega_{\mathcal{A}} \otimes \Omega_{\mathcal{B}})$.
\end{definition}
\begin{rmk} 
Note that since $K \otimes K \cong K$, $\Omega_M \otimes \Omega_N: M \otimes N \rightarrow K$.
\end{rmk}
\begin{theorem}
Let $\mathcal{A} = (M,B,s_{\mathcal{A}},\triangleright_{\mathcal{A}},
\Omega_{\mathcal{A}})$ and 
$\mathcal{B}=(N,B,s_{\mathcal{B}},\triangleright_{\mathcal{B}},
\Omega_{\mathcal{B}})$ be left $K$-linear automata.  
Then $\rho_{\mathcal{A} \otimes \mathcal{B}} = \rho_{\mathcal{A}} * \rho_{\mathcal{B}}$.
\end{theorem}
\begin{proof}
For any $b \in B$,
\begin{align*}
\rho_{\mathcal{A} \otimes \mathcal{B}}(b) &= \Omega_{\mathcal{A} \otimes \mathcal{B}}(b \triangleright_{\mathcal{A} \otimes \mathcal{B}}(s_{\mathcal{A}} \otimes s_{\mathcal{B}}))\\
&= \Omega_{\mathcal{A} \otimes \mathcal{B}}(\sum_i b_{(1)} \triangleright_{\mathcal{A}} s_{\mathcal{A}} \otimes b_{(2)} \triangleright_{\mathcal{B}}
 s_{\mathcal{B}})\\
&= \sum_i \Omega_{\mathcal{A}}(b_{(1)}\triangleright_{\mathcal{A}} s_{\mathcal{A}}) \Omega_{\mathcal{B}}(b_{(2)}
\triangleright_{\mathcal{B}} s_{\mathcal{B}})\\
&= \rho_{\mathcal{A}} * \rho_{\mathcal{B}}(b).\\
\end{align*}
\end{proof} 
\noindent In the classical case, this corresponds to ``running two automata in parallel''.
\begin{example}
Consider the following automata:
$$
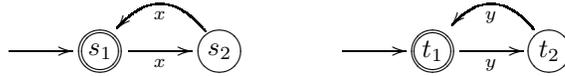
\xymatrix{\ar[r]
& *++[o][F=]{s_1}\ar[r]_{x}
& *++[o][F]{s_2}\ar@/_1.5pc/[l]^{x}
& }
\xymatrix{\ar[r]
& *++[o][F=]{t_1}\ar[r]_{y}
& *++[o][F]{t_2} . \ar@/_1.5pc/[l]^{y} 
& }
$$
They accept the languages denoted by $(xx)^*$ and $(yy)^*$, respectively.
We provide the tensor product of the $K$-algebraic encodings of these automata,
using the comultiplication $\Delta_3$. We assume that both automata have input algebra $K\{x,y\}^*$;
the action of $y$ on the $K$-semimodule of the first automaton is given by the
$2 \times 2$ matrix of 0's, as is the action of $x$ on the $K$-semimodule of the
second.

The $K$-semimodule of the tensor product is the free $K$-semimodule on
the set $\{s_1 \otimes t_1, s_1 \otimes t_2, s_2 \otimes t_1, s_2 \otimes
t_2\}$, by  Lemma \ref{lemma:tensorproperties}.6. The start vector is
$$
\left[
\begin{array}{llll}
1 & 0 & 0 & 0 \\ \end{array}
\right],
$$
the right $x,y$ actions are given by
$$
\left[
\begin{array}{llll}
0 & 0 & 1 & 0 \\
0 & 0 & 0 & 1 \\
1 & 0 & 0 & 0 \\
0 & 1 & 0 & 0\\ \end{array}
\right],
\left[
\begin{array}{llll}
0 & 1 & 0 & 0 \\
1 & 0 & 0 & 0 \\
0 & 0 & 0 & 1 \\
0 & 0 & 1 & 0\\ \end{array}
\right]
$$
respectively, and the observation function is given by
$$
\left[
\begin{array}{llll}
k_1 & k_2 & k_3 & k_4 \\
\end{array}
\right] \cdot
\left[
\begin{array}{l}
1 \\
0 \\
0 \\
0 \\ 
\end{array}
\right].
$$
\end{example}

\section{$K$-linear Automata and Deterministic Automata}
\label{section:determinize}
We now define deterministic automata and relate deterministic automata to
$K$-linear automata.  We treat only right automata; the left automata case is similar.
\subsection{Deterministic Automata}
Let the symbol $1$ denote a canonical one-element set.  
\begin{definition}
A {\em right deterministic automaton} $D = (S,\Sigma,\alpha,\delta,\Omega,O)$
consists of:
\begin{enumerate}
\item A set $S$ of states, 
\item An input alphabet $\Sigma$,
\item A {\em start function} $\alpha: 1 \rightarrow S$,
\item A {\em transition function} $\delta: \Sigma \rightarrow (S \rightarrow
S)$,
\item A set $O$ of outputs and an {\em output function} $\Omega:S \rightarrow
O$.
\end{enumerate}
\end{definition}
We use ``rightness'' to extend the domain of $\delta$ from $\Sigma$
to $\Sigma^*$. Let $\Endr(S)$ be the monoid consisting of all
functions $S \rightarrow S$ with composition defined on the right.  By
freeness of $\Sigma^*$, $\delta$ can be uniquely extended to a monoid
homomorphism $$\delta_w: \Sigma^* \rightarrow \Endr(S).$$ 
Using $\delta_w$, we define the language accepted by $D$.
\begin{definition}
Let $D$ be a deterministic automaton.  The {\em language accepted} by $D$ is
the function
$$\rho: \Sigma^* \rightarrow O$$
$$\rho(w) = \Omega(\delta_w(\alpha(1))).$$
\end{definition}
Of special importance are maps between automata which preserve the language
accepted.
\begin{definition}
Let $D = (S,\Sigma,\alpha_D,\delta_D,\Omega_D,O)$ and $E=
(T,\Sigma,\alpha_E,\delta_E,\Omega_E,O)$ be deterministic automata.  
A {\em deterministic automaton morphism} is a map $$f: S \rightarrow T$$
such that the following diagrams commute:
$$
\xymatrix{
1 \ar[r]^{\alpha_D} \ar[dr]_{\alpha_E} & S\ar[d]^{f} &
S \ar[r]^{\delta_D} \ar[d]_{f} & S \ar[d]^{f} &
S \ar[r]^{\Omega_D}  \ar[d]_{f}& O\\
& T &
T \ar[r]_{\delta_E} & T &
T. \ar[ur]_{\Omega_E} 
}
$$
\end{definition}

If such a map exists, then $\rho_D(w) = \rho_E(w)$ for all $w \in \Sigma^*$;
the proof is essentially the same as the proof of Theorem
\ref{theorem:samelanguage}.  As with $K$-linear automata, deterministic
automata and deterministic automaton morphisms form a category.

Given an automaton $D$, we can remove states that don't contribute to $\rho_D$.
\begin{definition}
Let $D = (S,\Sigma,\alpha,\delta,\Omega,O)$ be a deterministic automaton.
A state $s \in S$ is {\em accessible} if there exists a $w \in \Sigma^*$
such that
$$\delta_w(\alpha(1)) = s.$$ 
\end{definition}

\begin{definition}
Let $D = (S,\Sigma,\alpha,\delta,\Omega,O)$ be a deterministic automaton.
Let $S'$ be the set of accessible states of $D$ and let $i$ be the
inclusion $S' \rightarrow S$.  The {\em accessible subautomaton} of $D$ is the
automaton $D' = (S',\Sigma,\alpha,\delta \circ i, \Omega \circ i, O)$.
\end{definition}
\begin{lemma}
\label{lemma:accessible}
Let $D = (S,\Sigma,\alpha,\delta,\Omega,O)$ be a deterministic automaton and
let $D'$ be its accessible subautomaton.  Then $\rho_D = \rho_{D'}$.
\end{lemma}
\begin{proof}
The inclusion $S' \rightarrow S$ is a deterministic automaton morphism.
\end{proof}
A useful property of deterministic automata is
that they can be minimized.  This is a consequence of a certain category having
a final object; we must first tweak a definition.
\begin{definition}
A {\em deterministic labelled transition system} (dlts) $D =\\
(S,\Sigma,\delta,\Omega,O)$ is a deterministic automaton without a specified
start state. A {\em deterministic labelled transition system morphism}
is defined as a deterministic automaton morphism without the condition on the
start state.
\end{definition}
\begin{definition}
Let $D = (S,\Sigma,\delta,\Omega,O)$ be a dlts, 
and let $s \in S$.  The {\em language accepted} by $s$ is the function 
$$L_s(w): \Sigma^* \rightarrow O$$
$$L_s(w) = \Omega(\delta_w(s)).$$
\end{definition}
\begin{theorem}
Let $\Sigma$ be an alphabet and $O$ be a set of outputs.  Let $\mathfrak{C}$ be
the category of dlts's with input alphabet
$\Sigma$ and output set $O$, and morphisms thereof.  Then $F  =
(S,\Sigma,\delta,\Omega,O)$ is a final object of $\mathfrak{C}$, where
\begin{enumerate}
  \item $S = O^{\Sigma^*}$,
  \item $\delta(\psi)(w) = \psi(xw)$ for $\psi \in O^{\Sigma^*}, x \in
  \Sigma, w \in \Sigma^*$,
  \item $\Omega(\psi) = \psi(\lambda)$, for $\psi \in O^{\Sigma^*}$. 
\end{enumerate}
\end{theorem}
\begin{proof}
See Section 10 of \cite{bib:ucatos} (also the references
contained therein).  Given a dlts $D$,
the unique morphism $D \rightarrow F$ is $s \mapsto L_s$ for $s \in S_D$.
In the classical case, $F$ is the dlts with a state for each formal
language $L \subseteq \Sigma^*$ and transitions given by
Brzozowski derivatives.
\end{proof}
\begin{definition}
Let $D =(S,\Sigma,\alpha, \delta,\Omega,O)$ be a deterministic automaton
with all states accessible. The {\em minimization} of $D$, denoted $M(D)$, 
is the deterministic automaton obtained by the following procedure:
\begin{enumerate}
  \item Construct the underlying dlts $D'$ by ignoring
  the start function $\alpha$.
  \item Map $D'$ to $F$ via the unique morphism $f:s \mapsto L(s).$
  \item $M(D) = f(D')$ endowed with start state $f(\alpha_D(1))$.  The dlts
  morphism $f$ enriched with start state information is the unique deterministic
  automaton morphism $D \rightarrow M(D)$.
\end{enumerate} 
\end{definition}
This definition is justified in \cite{bib:aacaeic}.  The morphism
$D \rightarrow M(D)$ is, in particular, a function from the state set $S_D$
to the state set $S_{M(D)}$.  Any $D,D'$ which accept the same language
map to the same $M(D)$ by definition, so $|S_{M(D)}| \leq |S_D|$ (this is true
even if the automata involved have infinitely many states).

\subsection{K-linear Automata to Deterministic Automata}
Let $\mathcal{A} = (M,A,\alpha,\triangleleft, \Omega)$ be a $K$-linear
automaton. We wish to construct a deterministic automaton $D$ which is in some sense equivalent to $A$. 
This is possible using the notion of an adjunction between categories.   There
are many equivalent definitions of adjunctions used in practice, we recall the one
most useful for our purposes.
\begin{definition}
Let $\mathfrak{A}$ and $\mathfrak{D}$ be categories, $F$ a functor from
$\mathfrak{D}$ to $\mathfrak{A}$, and $U$ a functor from $\mathfrak{A}$ to
$\mathfrak{D}$. An {\em adjunction} from $\mathfrak{D}$ to $\mathfrak{A}$ is a
bijection $\psi$ which assigns to each arrow $f:F(D) \rightarrow A$ of 
$\mathfrak{A}$ an arrow $\psi f: D \rightarrow U(A)$ of $\mathfrak{D}$ such that 
$$\psi (f \circ Fh) = (\psi f) \circ h,$$
$$\psi (k \circ f) = Uk \circ (\psi f)$$
holds for all $f$ and all arrows $h:D' \rightarrow D$ and $k: A \rightarrow
A'$. Equivalently, for every arrow $g: D \rightarrow U(A)$,
$$\psi^{-1}(gh) = \psi^{-1}g \circ (Fh),$$
$$\psi^{-1}(Uk \circ g) = k \circ (\psi^{-1}g)$$
(omitting unnecessary parentheses).
\end{definition}
\begin{example}
Note that we use the notation of this example throughout the sequel.
Let $U'$ be the forgetful functor from \Kmod \text{ }to {\bf Set} and $F'$
the corresponding free functor.  The adjunction $\theta$ from
{\bf Set} to \Kmod \text{ }takes as input a $K$-linear
map $\phi:F'(X) \rightarrow M$ and returns the set map
$X \rightarrow U'(M)$ obtained by restricting $\phi$ to $X$.
\end{example}
Our goal is to construct a ``determinizing'' functor from a category
of $K$-linear automata to a category of deterministic automata, and a
``free $K$-linear'' functor in the opposite direction, and then to
show that these two functors are related by an adjunction.
In order for this to work nicely, we make the following assumptions.
\begin{enumerate}
  \item The input $K$-algebra of the $K$-linear automata is 
  the free $K$-algebra on a finite set $\Sigma$.
  \item The input alphabet of the deterministic automata is $\Sigma$,
  and the output set of the deterministic automata is the underlying 
  set of $K$.
\end{enumerate}
When considering start functions, we treat $K$ as $F'(1)$.

Let $\mathfrak{A}$ be a category of $K$-linear automata and $K$-linear
automaton morphisms, satisfying assumption 1 above, and let $\mathfrak{D}$
be a category of deterministic automata and deterministic automaton morphisms,
satisfying assumption 2.  We define a functor $U$ from $\mathfrak{A}$ to
$\mathfrak{D}$ which in the classical case corresponds to determinization
via the subset construction.

On $K$-linear automata, $U$ behaves as follows.  Given a $K$-linear automaton 
$\mathcal{A} = (M,K\Sigma^*,\alpha,\triangleleft, \Omega)$, 
$$U(\mathcal{A}) = (U'(M),\Sigma, \theta(\alpha), \delta, U'(\Omega), U'(K)),$$
where $\delta$ is defined as follows.  The
action $M \triangleleft K\Sigma^*$ is equivalent to a $K$-algebra map
$$K\Sigma^* \rightarrow \Endr(M).$$
Restricting this action to the generators of $K\Sigma^*$ yields 
a map $t$ from $\Sigma$ to the right endomorphism monoid of $M$; define
$\delta(x) = U'(t(x))$.

We now define $U$ on arrows of $\mathfrak{A}$. 
Let $\mathcal{A} =
(M,K\Sigma^*,\alpha_{\mathcal{A}},\triangleleft_{\mathcal{A}}, 
\Omega_{\mathcal{A}})$ and $\mathcal{B} =
(N,K\Sigma^*,\alpha_{\mathcal{B}},\triangleleft_{\mathcal{B}},
\Omega_{\mathcal{B}})$ be $K$-linear automata. A $K$-linear automaton morphism
$\phi: \mathcal{A} \rightarrow \mathcal{B}$ is, in particular, a $K$-linear map
$M \rightarrow N$. Define $U(\phi)$ to be the underlying set map $U'(\phi)$.  To
show that $U$ takes morphisms of $\mathfrak{A}$ to morphisms of $\mathfrak{D}$,
we must show that the commutativity of 
 $$
\xymatrix{
F'(1) \ar[r]^{\alpha_{\mathcal{A}}} \ar[dr]_{\alpha_{\mathcal{B}}} &
M\ar[d]^{\phi} & M \ar[r]^{\triangleleft_{\mathcal{A}}} \ar[d]_{\phi} & M
\ar[d]^{\phi} & M \ar[r]^{\Omega_{\mathcal{A}}}  \ar[d]_{\phi}& K\\
& N &
N \ar[r]_{\triangleleft_{\mathcal{B}}} & N &
N \ar[ur]_{\Omega_{\mathcal{B}}} 
}
$$
implies the commutativity of 
$$
\xymatrix{
1 \ar[r]^{\theta(\alpha_{\mathcal{A}})~~}
\ar[dr]_{\theta(\alpha_{\mathcal{B}})} & 
U'(M)\ar[d]^{U'(\phi)} & U'(M) \ar[r]^{\delta} \ar[d]_{U'(\phi)} & U'(M) \ar[d]^{U'(\phi)} &
U'(M) \ar[r]^{U'(\Omega_{\mathcal{A}})}  \ar[d]_{U'(\phi)}& U(K)\\
& U'(N) &
U'(N) \ar[r]_{\delta} & U'(N) &
U'(N). \ar[ur]_{U'(\Omega_{\mathcal{B}})} 
}
$$
The transition and output diagrams commute 
because the functor $U'$ takes commutative diagrams to commutative
diagrams.  To show that the start function diagram commutes, 
note that 
$$\theta (\phi \circ \alpha_{\mathcal{A}}) = U'(\phi)\circ
\theta(\alpha_{\mathcal{A}})$$ since $\theta$ is an adjunction.  Since
$\alpha_{\mathcal{B}} = \phi \circ \alpha_{\mathcal{A}},$ we have
$\theta(\alpha_{\mathcal{B}}) = U'(\phi)\circ \theta(\alpha_{\mathcal{A}})$.
\begin{theorem}
\label{theorem:uisafunctor}
The function $U$ is a functor from $\mathfrak{A}$ to $\mathfrak{D}$.
\end{theorem}
\begin{proof}
We have given the action of $U$ on objects and morphisms of $\mathcal{A}$.
It remains to show that 
$$U(1_{\mathcal{A}}) = 1_{U({\mathcal{A}})},$$
$$U(\phi' \circ \phi) = U(\phi') \circ U(\phi).$$
This is the case because $U$ is the restriction of the functor $U'$ to
$K$-linear maps which are also $K$-linear automaton morphisms.
\end{proof}
The following theorem follows easily from the definitions.
\begin{theorem}
\label{theorem:sameminimal}
Let $\mathcal{A}$ be a $K$-linear automaton.  Then $\theta(\rho_{\mathcal{A}}) =
\rho_{U(\mathcal{A})}$.
\end{theorem}
\begin{rmk}
Depending on $K$, it is possible for $U$ to take a $K$-linear automaton
whose underlying $K$-semimodule is the free $K$-semimodule on a finite set $X$
and return a deterministic automaton with infinitely many states.  This is not 
surprising; if the range of the language accepted by
a deterministic automaton $D$ is infinite, then $D$ must have
infinitely many states.  Furthermore, even in the classical case, it well-known that there 
are nondeterministic automata with $n$ states such that any equivalent deterministic 
automaton requires a number of states exponential in $n$. In other words, a
$K$-semimodule structure can be a significant asset to computation.
\end{rmk}
\subsection{Deterministic Automata to K-linear Automata}
We now define a functor $F: \mathfrak{D} \rightarrow \mathfrak{A}$.  In the
classical case, this functor is used implicitly when encoding a deterministic
automaton using matrices. 

Given a deterministic automaton $D= (S,\Sigma,\alpha,\delta,\Omega,U'(K))$,
the free $K$-linear automaton $F(D)$ is 
$$
(F'(S),K\Sigma^*,F'(\alpha), \triangleleft, \theta^{-1}(\Omega))
$$
where $\triangleleft$ is defined as follows. Apply $F'$ to $\delta(x)$ for each
$x \in \Sigma$. This yields a map from $\Sigma$ to $\Endr(F'(S))$,
which has a unique extension to an algebra map $K\Sigma^* \rightarrow
\Endr(F'(S))$.

Let $D = (S,\Sigma,\alpha_D,\delta_D,\Omega_D,U'(K))$ and $E=
(T,\Sigma,\alpha_E,\delta_E,\Omega_E,U'(K))$ be deterministic automata,
and $f$ a morphism $D \rightarrow E$.  Define $F(f) = F'(f)$; we must show that
$F'(f): F'(S) \rightarrow F'(T)$ is a $K$-linear automaton morphism
$F(D) \rightarrow F(E)$.  Dual to the determinizing case, it is easy to
see that $F'(f)$ behaves well on the transition and input functions.
We must show that   
$$\theta^{-1}(\Omega_D) = \theta^{-1}(\Omega_E) \circ F'(f).$$
This follows from the equations 
$\theta^{-1}(\Omega_E \circ f) = \theta^{-1}(\Omega_E) \circ F'(f)$
and \\
$\Omega_E \circ f = \Omega_D.$
\begin{theorem}
The function $F$ defined above is a functor from $\mathfrak{D}$ to 
$\mathfrak{A}$.
\end{theorem}
\begin{proof}
Similar to the proof of Theorem \ref{theorem:uisafunctor}.
\end{proof}
\subsection {Adjunctions Between Categories of Automata} 
We now show that the functors $F$ and $U$ defined above are related
by an adjunction.  Let $D=(S,\Sigma,\alpha_D,\delta,\Omega_D,U'(K))$ be a
deterministic automaton and
$\mathcal{A}=(M,K\Sigma^*,\alpha_{\mathcal{A}},\triangleleft,\Omega_{\mathcal{A}})$
a $K$-linear automaton. We must find a bijection 
$$\psi: \mathfrak{A}(F(D),A) \rightarrow \mathfrak{D}(D,U(A))$$
such that the conditions of an adjunction are
satisfied. We claim that the desired $\phi$ is a restriction of the adjunction
between \Kmod \text{ } and {\bf Set}.  
\begin{lemma}
\label{lemma:restrict}
Let $D=(S,\Sigma,\alpha_D,\delta,\Omega_D,U'(K))$ be a deterministic
automaton,\\
$\mathcal{A}=(M,K\Sigma^*,\alpha_{\mathcal{A}},\triangleleft,\Omega_{\mathcal{A}})$
a $K$-linear automaton, and $\phi$ a $K$-linear automaton morphism $F(D) \rightarrow \mathcal{A}$.  Then $$\psi(\phi) = \phi|_S:D \rightarrow
U(\mathcal{A})$$ is a deterministic automaton morphism $D \rightarrow U(\mathcal{A})$.
\end{lemma}
\begin{proof}
By definition of $F$ and $U$, and the fact that $\phi$ is a $K$-linear automaton
morphism, the following diagrams commute:
$$
\xymatrix{
F'(1) \ar[r]^{F'(\alpha_D)} \ar[dr]_{\alpha_{\mathcal{A}}} & F'(S)\ar[d]^{\phi}
& F'(S) \ar[r]^{\delta} \ar[d]_{\phi} & F'(S) \ar[d]^{\phi} &
F'(S) \ar[rr]^{\theta^{-1}(\Omega_D)}  \ar[d]_{\phi}&  &K\\
& M &
M \ar[r]_{\triangleleft_{\mathcal{A}}} & M &
M. \ar[urr]_{\Omega_{\mathcal{A}}} 
}
$$
To show that $\psi(f)$ is a deterministic automaton morphism, we must show the
the commutativity of
$$
\xymatrix{
1 \ar[r]^{\alpha_D} \ar[dr]_{\theta(\alpha_A)} & S \ar[d]^{\psi(\phi)} &
S \ar[r]^{\delta_D} \ar[d]_{\psi(\phi)} & S \ar[d]^{\psi(\phi)} &
S \ar[r]^{\Omega_D}  \ar[d]_{\psi(\phi)}& U'(K)\\
& U'(M) &
U'(M) \ar[r]_{\delta} & U'(M) &
U'(M). \ar[ur]_{U'(\Omega_A)} 
}
$$
This can easily be shown by diagram chasing.
\end{proof}
Note that $\psi(\phi) = \theta(\phi)$, when $\phi$ is considered as a $K$-linear
map.
\begin{lemma}
\label{lemma:extend}
Let $D=(S,\Sigma,\alpha_D,\delta,\Omega_D,U'(K))$ be a deterministic
automaton,\\
$\mathcal{A}=(M,K\Sigma^*,\alpha_{\mathcal{A}},\triangleleft,\Omega_{\mathcal{A}})$
a $K$-linear automaton, and $f$ a deterministic automaton morphism $D \rightarrow U({\mathcal{A}})$.  Then 
$$\psi^{-1}(f) = F(D) \rightarrow \mathcal{A},$$ 
the $K$-linear extension of $f$, is a $K$-linear automaton morphism $F(D)
\rightarrow \mathcal{A}$.
\end{lemma}
\begin{proof}
Let $\phi = \psi^{-1}(f)$.  As in the proof of Lemma \ref{lemma:restrict}; it is
easy to see that the commutativity of
$$
\xymatrix{
1 \ar[r]^{\alpha_D} \ar[dr]_{\theta(\alpha_{\mathcal{A}})} & S \ar[d]^{f} &
S \ar[r]^{\delta_D} \ar[d]_{f} & S \ar[d]^{f} &
S \ar[r]^{\Omega_D}  \ar[d]_{f}& U'(K)\\
& U'(M) &
U'(M) \ar[r]_{\delta} & U'(M) &
U'(M) \ar[ur]_{U'(\Omega_{\mathcal{A})}} 
}
$$
implies the commutativity of 
$$
\xymatrix{
F'(1) \ar[r]^{F'(\alpha_D)} \ar[dr]_{\alpha_{\mathcal{A}}} & 
F'(S)\ar[d]^{\phi} &  F'(S) \ar[r]^{\delta} 
\ar[d]_{\phi} & F'(S) \ar[d]^{\phi} & F'(S)
\ar[rr]^{\theta^{-1}(\Omega_D)} \ar[d]_{\phi}& & K\\ & M &
M \ar[r]_{\triangleleft_{\mathcal{A}}} & M &
M. \ar[urr]_{\Omega_{\mathcal{A}}} 
}
$$
\end{proof}
\begin{theorem}
\label{theorem:adjoint}
The function $\psi$ is an adjunction from $\mathfrak{D}$ to $\mathfrak{A}$.
\end{theorem}
\begin{proof}
Lemmas \ref{lemma:restrict} and \ref{lemma:extend} imply that $\psi$ is a
bijection between $\mathfrak{A}(F(D),A)$ and $\mathfrak{D}(D,U(A))$.
Furthermore, $\psi$ is the restriction of the adjunction between \Kmod \text{ }
and {\bf Set} to $K$-linear maps which are also automaton morphisms.  
For all arrows $k: A \rightarrow A'$ in $\mathfrak{A}$ and
$h: D' \rightarrow D$ in $\mathfrak{D}$, we have $Uk = U'k$
and $Fh = F'h$. Therefore
$$\psi (\phi \circ Fh) = \psi \phi \circ h,$$
$$\psi (k \circ \phi) = Uk \circ \psi \phi$$
for all arrows $\phi: F(D) \rightarrow A$.
\end{proof}
\section{Automaton Morphisms as Equivalence Proofs}
\label{section:proofs}

By Theorem \ref{theorem:samelanguage}, $K$-linear automaton morphisms 
preserve the language accepted by an automaton.  This can be thought
of as a soundness proof for a proof system for $K$-linear automaton equivalence
in which a proof consists of a sequence of $K$-linear automata and morphisms
between them.  We now show that given any two equivalent $K$-linear automata
$\mathcal{A}$ and $\mathcal{B}$, we can find a sequence of
$K$-linear automata and morphisms from $\mathcal{A}$ to $\mathcal{B}$; i.e.,
that the aforementioned proof system is complete.
\begin{theorem}
\label{theorem:sequence}
Let $\mathcal{A}$ be a $K$-linear automaton.  We have the following sequence
of $K$-linear automata and morphisms:
$$
\xymatrix{
\mathcal{A} & F(U(\mathcal{A})) \ar[l]_{\epsilon \quad}& & F(U(\mathcal{A})')
\ar[ll]_{F(i)} \ar[rr]^{F(m)} & & F(M(U(\mathcal{A})')) }
$$
\end{theorem}
\begin{proof}
The morphism from $F(U(\mathcal{A}))$ is the counit of the adjunction
$\psi$ between $\mathfrak{A}$ and $\mathfrak{D}$.  The deterministic 
automaton $U(\mathcal{A})'$ is the accessible subautomaton of $U(\mathcal{A})$
and $i$ is the inclusion of $U(\mathcal{A})'$ into $U(\mathcal{A})$.  
The deterministic automaton morphism $m$ is the morphism from $U(\mathcal{A})'$
to $M(U(\mathcal{A})')$, the minimization of $U(\mathcal{A})'$.
\end{proof}
\begin{rmk}
The above sequence can be shortened since $\epsilon \circ F(i)$ is 
a morphism from $F(U(\mathcal{A})')$ to $\mathcal{A}$.
\end{rmk}
\begin{corollary}
Let $\mathcal{A}$ and $\mathcal{B}$ be equivalent right $K$-linear automata.
There is a sequence of $K$-linear automata and morphisms which witness the
equivalence.
\end{corollary}
\begin{proof}
By Theorem \ref{theorem:sameminimal}, $U(\mathcal{A})$ and $U(\mathcal{B})$ are
equivalent deterministic automata, and therefore have the same minimization.
Applying Theorem \ref{theorem:sequence} to $\mathcal{A}$ and $\mathcal{B}$
yields sequences with the same endpoint; paste them together.
\end{proof}
\begin{rmk}
Theorem \ref{theorem:sequence} also holds for $K$-linear automata over arbitrary 
semirings, with some slight modifications.  In this case, we do not have an
algebra $K\Sigma^*$, but we can adjust the definition of a $K$-linear automaton 
to compute a map $\Sigma^* \rightarrow K$.
\end{rmk}
If the above sequence can be represented finitely, then one can ask questions
about the complexity of the proof system.  In \cite{bib:apgka}, it is shown that
such a sequence can be produced by a $PSPACE$ transducer for classical finite 
nondeterministic automata.  The morphisms can be represented by $|\Sigma|$ many
matrices; if the linear intertwining condition holds for the generators of the algebra,
it holds for the entire $K$-algebra.

\section{Acknowledgements}
The author would like to thank Anil Nerode for many inspiring discussions and the anonymous 
reviewers of LFCS '09 for their helpful comments and suggestions.  This work was
supported by NSF grant CCF-0635028.

%\bibliography{elsarticle-num}

\begin{thebibliography}{99}

\bibitem{bib:aac} Ad\`{a}mek, J. and Trnkov\`{a}, V.
\rm{Automata and Algebras in Categories}. {\em Kluwer Academic Publishers.}
1990.

\bibitem{bib:racom} Anderson, Frank W. and Fuller, Kent R.
\rm{Rings and Categories of Modules.}
{\em Springer-Verlag.}  1992.

\bibitem{bib:brwa} Buchholz, P.
\rm{Bisimulation Relations for Weighted Automata}.
{\em Theoretical Computer Science}, 393:109-123.  2008.

\bibitem{bib:ddlawm} Duchamp, G., Flouret, M., Laugerotte, \`{E}, and Luque, J.-G.
\rm{Direct and Dual Laws for Automata with Multiplicities.} 
{\em Theoretical Computer Science}.  267:105-120. 2001.

\bibitem{bib:sdsc} Duchamp, G. and Tollu, Christophe.
Sweedler's Duals and Sch\"{u}tzenberger's Calculus.
{\rm Arxiv Preprint}.	arXiv:0712.0125v2.

\bibitem{bib:bbv} Fitting, Melvin.
\rm{Bisimulations and Boolean Vectors.}
{\rm Advances in Modal Logic.} Volume 4:97-125.  2003.

\bibitem{bib:sata} Golan, Jonathan S.
\rm{Semirings and Their Applications.} {\em Kluwer Academic Publishers.} 1999

\bibitem{bib:bar} Grossman, R.L. and Larson, R.G.
{\rm Bialgebras and Realizations}. In Hopf Algebras: J. Bergen, S. Catoiu, and W. Chin, eds. pp 157-166.
{\it Marcel Dekker, Inc.} 2004.

\bibitem{bib:riob} Grossman, R.L. and Larson, R.G.
The Realization of Input-Output Maps Using Bialgebras.
{\it Forum Mathematicum.} Volume 4, pp. 109-121, 1992.  

\bibitem{bib:tpie} Katsov, Yefim.
\rm{Tensor Products and Injective Envelopes of Semimodules over Additively Regular Semirings}.
{\it Algebra Colloquium} 4:2 121-131. 1997.

\bibitem{bib:ko94} Kozen, Dexter. 
\rm{A Completeness Theorem for Kleene Algebras and the Algebra of Regular Events}.
{\it Infor. and Comput}, 110(2):366-390. May 1994.

\bibitem{bib:lang} Lang, Serge.
\rm{Algebra: Revised Third Edition}.
{\it Springer-Verlag}. 2002.

\bibitem{bib:tpis} Litvinov, G.L., Masloc, V.P., and Shpiz, G.B.
\rm{Tensor Products of Idempotent Semimodules.  An Algebraic Approach.}
{\it Mathematical Notes}. Vol 65, No. 4, 1999.

\bibitem{bib:fqgt} Majid, Shahn.
\rm{Foundations of Quantum Group Theory.}  {\it Cambridge University Press}. 1995

\bibitem{bib:aacaeic} Rutten, J.J.M.M.
{\rm Automata and Coinduction(An Exercise in Coalgebra).}
In {\it Proc. CONCUR '98}, {\rm volume 1466 of} {\it LNCS}, {\rm pages 194-218}.
Springer-Verlag, 1998.

\bibitem{bib:ucatos} Rutten, J.J.M.M.
{\rm Universal Coalgebra: A Theory of Systems}.
{\it Theoretical Computer Science}. 249 pp. 3-80.  2000.

\bibitem{bib:qgpca} Street, Ross.
\rm{Quantum Groups: A Path to Current Algebra.} {\it Cambridge University Press}. 2007

\bibitem{bib:apgka} Worthington, James.
Automatic Proof Generation in Kleene Algebra.  In R. Berghammer, B. M\"{o}ller, and G. Struth, editors,
{\it 10th Int. Conf. Relational Methods in Computer Science (RelMiCS10) and 5th Int. Conf Applications
of Kleene Algebra (AKA5)}, {\rm volume 4988 of} \it{LNCS}, {\rm pages 382-396}. Springer-Verlag, 2008.

\end{thebibliography}

\end{document}